\documentclass[a4paper,10pt]{amsart}
\usepackage[utf8]{inputenc}
\usepackage{amsmath,amssymb,amsthm}
\usepackage{xcolor}
\usepackage{enumerate}
\usepackage{tikz}
\usepackage{tikz-cd}
\usepackage[all]{xy}
\usepackage{float}
\usepackage{soul}
\usetikzlibrary{decorations.pathreplacing}

\usepackage{listings}

\usepackage{hyperref}
\hypersetup{hidelinks}

\usepackage[noabbrev,capitalise,nameinlink]{cleveref}
\crefname{subsection}{Subsection}{Subsections}

\usepackage{thmtools}

\usepackage{stmaryrd}

\theoremstyle{plain}

\newtheorem{thm}{Theorem}[section]
\newtheorem{corollary}[thm]{Corollary}
\newtheorem{lemma}[thm]{Lemma}
\newtheorem{question}[thm]{Question}
\newtheorem{proposition}[thm]{Proposition}

\newtheorem*{conjecture*}{Conjecture}

\theoremstyle{definition}

\newtheorem{definition}[thm]{Definition}

\DeclareMathOperator{\soc}{soc}
\DeclareMathOperator{\aut}{Aut}

\title{Sectionally indecomposable groups}

\author{Andrea Lucchini}
\address{Andrea Lucchini. University of Padova (Italy), Dipartimento di Matematica ``Tullio Levi Civita''. ORCID: https://orcid.org/0000-0002-2134-4991}
\email{lucchini@math.unipd.it}

\author{Nowras Otmen} 
\address{Nowras Otmen. University of Padova (Italy), Dipartimento di Matematica ``Tullio Levi Civita''. ORCID: https://orcid.org/0009-0009-8092-1689}
\email{nowras.naufel@math.unipd.it}

\date{\today}

\begin{document}

\maketitle

\begin{abstract}
We introduce the notion of sectional indecomposability and study it for finite groups: a group $H$ is sectionally indecomposable if, whenever $H$ is a section of a direct product $A \times B$, then $H$ is already a section of $A$ or of $B$. We show that the study of sectionally indecomposable finite groups reduces to the monolithic case. Our main result is a complete characterisation of sectional indecomposability for monolithic primitive groups: such a group $G$ with $N = \mathrm{soc}(G)$ is sectionally indecomposable if and only if either $N$ is non-abelian, or $N$ is a $p$-group and  $O_{p'}(G/N) \neq 1$. The proof relies on the introduction of the notion of an $H$-Frattini module and on the theory of the universal $p$-Frattini cover, together with a result of Griess--Schmid. As a corollary, every monolithic primitive solvable group is sectionally indecomposable. We also discuss the non-primitive case, which appears significantly harder, and highlight open questions concerning monolithic $p$-groups.
\end{abstract}

\section{Introduction}

There exist various notions of `indecomposability' in the theory of groups. Two illustrations of this are as follows. 

Let $G$ be a group. Then $G$ is \emph{directly indecomposable} if it is not isomorphic to a non-trivial direct product of groups. One way to rephrase this is that, if $G \cong H \times K$, then, without loss of generality, $G \cong H$ and $K$ is trivial. One can also speak of free indecomposability: we say that $G$ is \emph{freely indecomposable} if it does not split as a non-trivial free product of groups, and one may rephrase this similarly to what was done above for direct indecomposability.

The purpose of this note is to explore another concept which falls into this type of framework. Let $G$ and $H$ be two groups. Then $H$ is a \emph{section} of $G$ if it is the quotient of a subgroup of $G$ and we denote it by $H\leq_s G$. We define a group $H$ to be \emph{sectionally indecomposable} if, whenever $H \leq_s A \times B$ for groups $A$ and $B$, then either $H \leq_s A$ or $H \leq_s B$.

Although this definition makes sense in general, the bulk of this work concerns itself solely with finite groups, so this will be tacitly assumed until the last subsection, where we briefly comment on this property for infinite groups. As such, one can very quickly see that the study of sectionally indecomposable finite groups is reduced to the monolithic case, as shown in \cref{l:monolithic}. With this in mind, the main result of this work is a characterisation of sectional indecomposability for monolithic primitive groups, that is, those with trivial Frattini subgroup.

\begin{thm}\label{t:main}
Let $G$ be a primitive monolithic group and let $N=\soc(G).$ Then $G$ is sectionally indecomposable if, and only if, either $N$ is non-abelian or $N$ is a $p$-group and $O_{p^\prime}(G/N)\neq1.$
\end{thm}

This result is surprising for a couple of reasons. The first criterion to detect sectional indecomposability that we proved, which is the content of \cref{t:char-frath}, is very technical and a priori very hard to check. To obtain the main result above, we had to make a detour via the notions of the $p$-Frattini module and of the universal $p$-Frattini cover, the latter being a virtually pro-$p$ group that enjoys many special properties that are central to our proofs. These objects have been largely studied and exploited in contexts quite orthogonal to this work, as can be seen for instance in \cite{fried}.

Furthemore, drawing these connections in order to obtain the improved characterisation in \cref{t:main} has immediate consequences which could not have been made clear otherwise: for instance, it shows that \emph{every} monolithic primitive solvable group is sectionally indecomposable.

It follows immediately from the proof of \cref{l:monolithic} that if $G$ is not monolithic, then $G$ is neither sectionally indecomposable nor subdirectly indecomposable, where the latter means that whenever $G$ is a subdirect product of two groups, $G$ is a section of one of the two factors. However, the two notions are not equivalent when $G$ is monolithic. Indeed, any finite monolithic group is subdirectly indecomposable (see \cref{p:monsub}).

Now, we make some comments on the proofs of \cref{t:char-frath} and \cref{t:main}, respectively. The non-trivial case to be considered is that of monolithic primitive groups with abelian socle, as the ones with non-abelian socle are handled immediately in \cref{p:mono_non-abelian}. Then the important step is to single out the correct definitions. For a group $H$ and an $H$-module $A$, we introduce in \cref{frath} the notion of an \emph{$H$-Frattini module}: $A$ satisfies this condition if there exists a group $G$ with a normal subgroup $N$ such that $N$ is $H$-equivalent to $A$ (\cref{defequiv}) and $C_G(N) \leq \Phi(G)$. Pictorially, the situation is as follows

\[
\begin{tikzpicture}
    \filldraw[black]
        (0,0)    circle [radius=1pt]
        (0,-1.2) circle [radius=1pt]
        (0,-2)   circle [radius=1pt]
        (0,-3)   circle [radius=1pt];

    \draw (0,0) -- (0,-3);

    \node[left] at (0,0) {$G$};
    \node[left] at (0,-1.2) {$C_G(N)$};
    \node[left] at (0,-2) {$N$};
    \node[left] at (0,-3) {$\{1\}$};

    \draw[decorate,decoration={brace,amplitude=5pt}]
        (0.25,0.05) -- (0.25,-1.25)
        node[midway,right=8pt] {$\cong H$};

    \draw[decorate,decoration={brace,amplitude=5pt}]
        (0.25,-1.95) -- (0.25,-3.05)
        node[midway,right=8pt] {$\cong A$};

    \draw[decorate,decoration={brace,amplitude=5pt}]
        (1.7,-1.25) -- (1.7,-3.05)
        node[midway,right=8pt] {$\leq \Phi(G)$};
\end{tikzpicture}
\]

\noindent along with a compatibility between the action of $G$ on $N$ and the action of $H$ on $A$. In the setting where $X=A \rtimes H$ is a monolithic primitive group with abelian socle, we prove in \cref{t:char-frath} that $X$ is not sectionally indecomposable if, and only if, $A$ is an $H$-Frattini module. For the `if' direction, if $A$ is $H$-Frattini, a group $G$ as above is just right in the sense that it easily allows us to find $X$ as a section of $G\times G$, while $G$ itself does not admit $X$ as a section. For the `only if' direction, we use crown theory to show that if $X$ is not sectionally indecomposable, we can find a $G$ as above, showing that $A$ is $H$-Frattini.

The issue with this characterisation is that it is in principle very hard to check. This is where the notion of the \emph{$p$-Frattini module $A_p(H)$} of a group $H$ for a prime $p$, due to Gäschütz \cite{ga54}, comes into play. It is such that there exists a Frattini extension $\widetilde H$ of $A_p(H)$ by $H$ (i.e. $A_p(H) \leq \Phi(\widetilde H)$) which is universal in the following sense: whenever $E$ is a Frattini extension of an $\mathbb F_p H$-module $B$ by $H$, then $E$ is an epimorphic image of $\widetilde H$ over $H$. This construction can be iterated, with each step of its iteration being recorded in one single object: the \emph{universal $p$-Frattini cover ${}_p \widetilde H$} of $H$. The details pertaining to these objects are taken up in \cref{subsec:fratti_mod}. Recall that if $X = A \rtimes H$ is a monolithic primitive group with abelian $\soc (X)=A$, then $A$ is an elementary abelian $p$-group for some prime $p$ which is faithful and irreducible as an $H$-module. When $A$ is $H$-Frattini, the group $G$ as above, after a suitable reduction, is a Frattini extension of a $p$-group by $H$, in such a way that $G$ will be an epimorphic image of ${}_p\widetilde H$. We leverage this connection and the work of Griess--Schmid, particularly \cite[Theorem 3]{GS78}, to prove the equivalence stated in \cref{t:main}.

This note is structured as follows. In \cref{sec:prelim}, we define sectional indecomposability and prove basic results related to it, recall the relevant facts of crown theory and of the $p$-Frattini module and prove auxiliary results in the context of the latter. In \cref{sec:monolithic}, we restrict ourselves to the study of sectional indecomposability in monolithic groups, with a focus on the case of primitive groups with abelian socle. In \cref{sec:main}, we prove \cref{t:main}. Finally, in \cref{sec:final}, we indicate how the non-primitive case seems a lot harder and is not in general amenable to the approaches developed here. In particular, the case of non-cyclic monolithic $p$-groups, equivalently, those with cyclic center, is highlighted. We prove that a few examples of $p$-groups are not sectionally indecomposable and leave open the question of the existence or not of a non-cyclic monolithic sectionally indecomposable $p$-group.

\section{Preliminaries}\label{sec:prelim}

\subsection{Sectionally indecomposable groups}

Here, we will define the notion of sectional indecomposability and prove some basic results. Let $G$ and $H$ be finite groups. One says that $G$ is a \emph{section} of $H$ if there exist $L \unlhd K \leq H$ such that $K/L \cong G$, and we denote this by $G \leq_s H$. 

\begin{definition}\label{d:strong_indec}
    A finite group $G$ is \emph{sectionally indecomposable} if, for finite groups $A$ and $B$, the condition $G \leq_s A \times B$ implies $G \leq_s A$ or $G \leq_s B$.
\end{definition}

The following easy characterisation of sectional indecomposability will be useful in a few proofs and in finding examples of groups which do not satisfy this property.

\begin{lemma}\label{l:strong_indec_subdirect}
    Let $G$ be a finite group. Then $G$ is sectionally indecomposable if, and only if, whenever a subdirect product $X \leq A \times B$ is such that $G$ is an epimorphic image of $X$, then $G \leq_s A$ or $G \leq_s B$.
\end{lemma}
\begin{proof}
    The direction $(\Rightarrow)$ is clear. If $G \leq_s A \times B$, then there exist $K \unlhd H \leq A \times B$ such that $H/K \cong G$. If we denote by $\pi_A$ and $\pi_B$ the projections of $A \times B$ onto $A$ and $B$, respectively, then $H \leq \pi_A(H) \times \pi_B(H)$ is subdirect and surjects onto $G$, so either $G \leq_s \pi_A(H)$ or $G \leq_s \pi_B(H)$, and we are done.
\end{proof}

Notice that, by induction, our definition of sectionally indecomposable is equivalent to the following statement: $G \leq_s A_1 \times \cdots \times A_n$ implies $G \leq_s A_i$ for some $i$. Recall that a finite group $G$ is \emph{monolithic} if it contains a unique minimal normal subgroup.

\begin{lemma}\label{l:monolithic}
	If a finite group $G$ is sectionally indecomposable, then $G$ is monolithic.
\end{lemma}

\begin{proof}
	If $G$ possesses two distinct minimal normal subgroups $N_1$ and $N_2$, then $G$ embeds in $G/N_1 \times G/N_2$, contradicting the sectional indecomposability of $G$.
\end{proof}

The proof of the previous lemma suggests a further definition: a finite group $G$ is \emph{subdirectly indecomposable} if, for finite groups $A$ and $B$, the condition $G \leq A \times B$ implies $G \leq_s A$ or $G \leq_s B$. It is easy to see that the following result holds.

\begin{proposition}\label{p:monsub} A finite group $G$ is subdirectly  indecomposable if and only if it is  monolithic. 
\end{proposition}
\begin{proof}
As we noticed in the proof of \cref{l:monolithic}, if $G$ possesses two distinct minimal normal subgroups $N_1$ and $N_2$, then $G$ embeds in $G/N_1 \times G/N_2$, hence it is not subdirectly indecomposable. Conversely, assume that $G\leq H_1\times H_2$, but neither $G\leq_s H_1$ nor $G\leq_s H_2.$ For $1\leq i\leq 2,$ consider the projection $\pi_i\colon G\leq H_1\times H_2\to H_i.$ Since $G/\ker \pi_i\cong G^{\pi_i}\leq H_i,$ it follows from $G\not\leq_s H_i$ that $\ker \pi_i\neq 1$. In particular, $\ker \pi_1$ and $\ker \pi_2$ are distinct non-trivial normal subgroups of $G$ with trivial intersection. Hence $G$ is not monolithic, as required.
\end{proof}

\cref{l:monolithic} immediately reduces the problem of determining whether a group $G$ is sectionally indecomposable to the case where $G$ is a monolithic group, which will be taken up in \cref{sec:monolithic}.

\begin{lemma}\label{l:H2}
Suppose a finite group $G$ can be written as $M \rtimes H$ with $M$ abelian and let $H^2(H,M)$ be the second cohomology group associated to this action. If there exists a non-split extension $E$ of $M$ by $H$ which defines a non-trivial element of $H^2(H,M)$ and $E\not\cong G$, then $G$ is not sectionally indecomposable.
\end{lemma}
\begin{proof}
Let $E$ be a non-split extension of $M$ by $H$ as in the hypothesis. Let $L = E \times E$ and, for any subgroup $U \leq E$, define
$$\mathrm{diag}(U) = {(u,u) \in E \times E : u \in U}.$$
Consider the subgroups $X = \mathrm{diag}(E)(M \times M)$ and $K = \mathrm{diag}(M) \unlhd X$ of $L$. Then
$X/K \cong (M \times M)/K \rtimes \mathrm{diag}(E)/K \cong M \rtimes H,$
where the action of $\mathrm{diag}(E)/K \cong H$ on $(M \times M)/K \cong M$ is induced by conjugation in $E$, hence it coincides with the initial action of $H$ on $M$. This implies $X/K \cong G$, and therefore $G \leq_s E \times E$. Since $E \not\cong G$, we have $G \not\leq_s E$, so $G$ is not sectionally indecomposable.
\end{proof}

\begin{corollary}\label{l:H2m}
   If $G=M\rtimes H$ is a primitive monolithic group,  $M=\soc(G)$ is abelian and $H^2(H,M) \neq 0$, then $G$ is not sectionally indecomposable.
\end{corollary}
\begin{proof}
If $H^2(H,M) \neq 0$, there exists a non-split extension $E$ of $M$ by $H$. Since $\Phi(G) = 1$ while $M \leq \Phi(E)$, $E$ and $G$ are not isomorphic and the conclusion follows from \cref{l:H2}.
\end{proof}

Suppose that $G$ is not sectionally indecomposable. We know that there exist $H_1, H_2$ which do not admit $G$ as a section, 
but such that $G \leq_s H_1 \times H_2$. One may ask whether, given that 
$G$ belongs to an assigned class $\mathcal{C}$, the subgroups $H_1$ and 
$H_2$ can also be chosen within that class. From the following lemma we 
will deduce that the answer is affirmative for many classes of groups, 
such as nilpotent, solvable, supersolvable groups, and $\pi$-groups, for 
any set $\pi$ of primes.

\begin{lemma}\label{class}Let $\mathcal C$ be a class of finite groups which satisfies the following properties:
	\begin{enumerate}
		\item if $Y \unlhd X$ and $X \in \mathcal{C}$, then $X/Y \in \mathcal{C}$;
		\item if $X/\Phi(X)\in \mathcal C,$ then $X\in \mathcal C.$
	\end{enumerate}
	If $N$ is a normal subgroup of a finite group $G$ and $G/N\in \mathcal C,$ then there exists $H\in \mathcal C$ such that $G=HN.$
\end{lemma}
\begin{proof}
	We prove the statement by induction on the order of the group. If $N\leq \Phi(G),$ then $G/\Phi(G)$ is an epimorphic image of $G/N$, so
	by (1) $G/\Phi(G) \in \mathcal C$ and therefore it follows from (2)
	that $G\in \mathcal C.$
	So we may assume $N\not\leq \Phi(G).$ In this case, there exists a maximal subgroup $M$ of $G$ such that $G=MN.$ Moreover, $M/M\cap N \cong MN/N= G/N \in  \mathcal C$ so there exists $H\in \mathcal C$ such that
	$M=(M\cap N)H$ and therefore $G=NM=N(M\cap N)H=NH.$
\end{proof}

\begin{corollary}Suppose that $\mathcal{C}$ is a class of finite groups satisfying the 
hypotheses of the previous lemma. If $G \in \mathcal{C}$ is not sectionally
indecomposable, then there exist $H_1, H_2 \in \mathcal{C}$ such that 
$G \leq_s H_1 \times H_2$, but $G$ is a section of neither $H_1$ nor $H_2$.
\end{corollary}

\begin{proof}Suppose $G\leq_s H_1\times H_2,$ and that $G$ is a section  of neither $H_1$ nor $H_2.$ We may assume that $G=X/N,$ with $X$ a subdirect product of $H_1\times H_2.$ By the previous lemma there exists $Y\in \mathcal C$ such that $X=YN$. In particular $G\cong Y/(Y\cap N).$ For $1\leq i\leq 2,$ consider the projection $\pi_i: H_1\times H_2\to H_i$. Then $K_i=Y^{\pi_i}\in \mathcal C$ and $G\leq_s K_1\times K_2.$ Moreover, since $K_i\leq H_i,$ $G$ is not a section of $K_i.$
\end{proof}

We note that the previous result does not hold in general if one does not assume
that the class $\mathcal{C}$ is such that if $G/\Phi(G) \in \mathcal{C}$,
then $G \in \mathcal{C}$. Consider, for example, for $p$ an odd prime,
the class $\mathcal{C}$ of finite $p$-groups of exponent $p$.
Let $G$ be the non-abelian group of order $p^3$ and exponent $p$.
We will see in \cref{pgruppi} that $G$ is not sectionally indecomposable. Nevertheless, suppose $G \leq_s H_1 \times H_2$ with $H_1, H_2 \in \mathcal{C}$.
Since $G$ is non-abelian, at least one of $H_1, H_2$ must be non-abelian; we may assume that $H_1$ is non-abelian. It is easy to see that any non-abelian group of exponent $p$ has a non-abelian quotient of order $p^3$, which must be isomorphic to $G$, so $G \leq_s H_1$.

\medskip

In order to prove our main results, we recall two notions that will be key to our proofs: that of a crown in a finite group and that of a $p$-Frattini module.

\subsection{Crowns in finite groups}

The notion of crowns was introduced by Gasch\"{u}tz  in \cite{praef} in the case of finite soluble groups and generalised in \cite{JL} to arbitrary finite groups.
A detailed exposition of the theory is also given in \cite[Section 1.3]{boez}.
Let $G$ be  a finite group, and $V$ an irreducible $G$-module. The monolithic primitive group $L_V$ associated to $V$ is defined as the semidirect product $L_V=V \rtimes G/C_G(V).$ It turns out that $L_V$ is an epimorphic image of $G$ if and only if $G$ has a complemented chief factor $G$-isomorphic to $V.$ Let $R_G(V)$ be the smallest normal subgroup contained in $C_G(V)$ and satisfying
the property that $C_G(V)/R_G(V)$ is complemented in $G/R_G(V)$ and isomorphic as a $G$-module to a direct product of copies of $V.$ It turns out that
$R_G(V)$ coincides with the intersection $\cap_{M \in \Delta(G,V)}M,$ where $\Delta(G,V)$ is the set of $G$-invariant
subgroups $M$ of $C_G(V)$ such that $C_G(V)/M$ is $G$-isomorphic to $V$ and $C_G(V)/M$ has a complement in $G/M.$ The
factor $C_G(V)/R_G(V)$ is called the $V$-crown of $G$. The number $\delta_G(V)$ such that $C_G(V)/R_G(V)\cong_G
V^{\delta_G(V)}$
is called the 
	$V$\nobreakdash-rank of $G$ and coincides with the number of
complemented factors $G$\nobreakdash-isomorphic to $V$ in any
chief series of $G.$ Notice in particular that if $N$ is a normal subgroup of $G$ and $G/N \cong L_V$ then $R_G(V)\leq N\leq C_G(V).$

\subsection{The $p$-Frattini module}\label{subsec:fratti_mod}

We recall a result from Gasch\"{u}tz \cite{ga54}. Let $G$ be a finite group, $p$ a prime, and $\mathbb{F}_p$ the field with $p$ elements. Then there is a finite $\mathbb{F}_pG$-module $A$ and a Frattini extension $\widetilde{G}$ of $A$ by $G$ (i.e. $A\leq \Phi(\widetilde{G})$), with the property that any other Frattini extension of a finite $\mathbb{F}_pG$-module by $G$ is an epimorphic image over $G$ of $\widetilde{G}$. We will call this module the \emph{$p$-Frattini module} of $G$, denoted by $A_p(G)$, and the extension $\widetilde G$ the \emph{universal $p$-Frattini extension} of $G$. From the definition, it can be seen immediately that if $p$ does not divide the order of $G$, then $A_p(G)$ is trivial. One can construct the $p$-Frattini module in many different ways. For instance, via properties of free groups in \cite[Theorem 11.8]{DH}; or if
\[\cdots \rightarrow P_1 \xrightarrow{\varepsilon_1} P_0 \rightarrow \mathbb F_p \rightarrow 0 \]
is the minimal projective resolution of the trivial $\mathbb F_pG$-module, then $A_p(G) = \ker \varepsilon_1$ in \cite{ga54} or \cite[Theorem $\beta$.7]{DH}. The latter yields interesting cohomological consequences, such as \cite[Lemma 3]{GS78}.

This process can be iterated. Let $G_0=G$ and $M_0 = A_p(G)$ and define $G_n = \widetilde{G}_{n-1}$ and $M_n = A_p(G_n)$. The $G_i$ form an inverse system whose inverse limit yields a profinite group ${}_p \widetilde{G}$, which is the \emph{universal Frattini $p$-cover} of $G$. It is a $p$-projective group (i.e., its $p$-cohomological dimension is 1) which fits into a short exact sequence of the form
\[1 \xrightarrow{} P \to {}_p \widetilde G \rightarrow G \rightarrow 1\]
\sloppy where $P$ is a free pro-$p$ subgroup contained in the Frattini of ${}_p\widetilde G$ and such that $\Phi^i(P)/\Phi^{i+1}(P) \cong M_i$ for $i\geq 0$. Here, $\Phi^0(P) = P$, $\Phi(P) = P^p[P,P]$ and $\Phi^i(P) = \Phi(\Phi^{i-1}(P))$. Notice that these are all characteristic subgroups of $P$, so they are normal in ${}_p\widetilde G$. In particular, if $\dim_{\mathbb F_p} A_p(G_0) = n$, then $P$ can be taken as the pro-$p$ completion $\widehat {F}_p$ of the free abstract group on $n$ generators, and so the inverse limit can be written as
\[\widehat F_p = \varprojlim \widehat F_p/\Phi^i(\widehat F_p) \cong \varprojlim F/\Phi^i(F).\]
In particular, any Frattini extension of $G$ with kernel a pro-$p$ group is an epimorphic image of ${}_p \widetilde G$ \cite[Proposition 25.7.1, Corollary 25.12.2]{FJ}. More details about the universal Frattini $p$-cover can be found in \cite[Section 25.12]{FJ}; its connection with $p$-Frattini modules is, for instance, in \cite[$\S$3]{semmen}.

We will also need the following notion. A finite group $G$ is \emph{$p$-supersolvable} for a prime $p$ if $G/O_{p',p}(G)$ is abelian with exponent dividing $p-1$. We will make repeated use of the following theorem due to Griess--Schmid, so we write it down for the reader's convenience.

\begin{thm}[{{\cite[Theorem 3]{GS78}}}]\label{t:cent_frattini}
    Let $G$ be a finite group whose order is divisible by $p$ and let $M_0=A_{p}(G)$. Then either $C_G(M_0) = O_{p'}(G)$ or $G$ is $p$-supersolvable with cyclic Sylow $p$-subgroups; the latter happens if and only if $\dim_{\mathbb F_p}M_0=1$ (and so $C_G(M_0)=O_{p',p}(G)$).
\end{thm}

A consequence of the above is that, when $G$ is $p$-supersolvable with cyclic Sylow $p$-subgroups, the kernel of ${}_p\widetilde G \to G$ is isomorphic to $\mathbb Z_p$; otherwise, it is a non-abelian free pro-$p$ group.

The following is hinted at, without proof, in \cite{CKK}.

\begin{proposition}\label{p:comp_factor}
    Let $H$ be a finite group with order divisible by $p$, let $_p \widetilde H$ be its universal $p$-Frattini cover and assume that $H$ is not $p$-supersolvable. If an irreducible $\mathbb F_pH$-module $S$ is centralised by $O_{p'}(H)$, then it is $H$-isomorphic to a chief factor of a finite quotient of $_p \widetilde H$ lying in one of the $M_n$.
\end{proposition}

\begin{proof} 
    Suppose $S$ is an irreducible $\mathbb F_pH$-module centralised by $O_{p'}(H)$ and let $A=A_p(H)$. By assumption and by \cref{t:cent_frattini}, the dimension of $A$ is greater than 1 and $C_H(A)=O_{p'}(H)$, so there is an embedding $\bar{H}=H/O_{p'}(H) \hookrightarrow \mathrm{GL}_n(\mathbb F_p)=\Sigma$ with $n=\dim_{\mathbb F_p}A$. The aim is to use the results in \cite{BK}, so we place ourselves within that context and follow their notation. For any group $G$, define $G_1 = G$ and $G_i = G_{i-1}^p[G_{i-1},G]$. Let $F$ be the free abstract group on $n$ generators. Given that $F/F_2$ is an elementary abelian $p$-group of rank $n$, we can choose a $\Sigma$-action on it in a way that $A$ and $F/F_2$ are isomorphic as $\mathbb F_p\Sigma$-modules. Combining \cite[Theorem 2]{BK} and \cite[Theorem 3]{BK} and the discussion therein, each $F_i/F_{i+1}$ is a $\mathbb F_p\Sigma$-module and, for sufficiently large $i$, say $i \geq i_0$, they contain a regular $\mathbb F_p \Sigma$-module. Since $\bar{H}$ is a subgroup of $\Sigma$, this means that for $i \geq i_0$ they also contain a regular $\mathbb F_p\bar{H}$-module. Now, by the construction above of the universal $p$-Frattini cover of $H$, we have
    \[1 \rightarrow \widehat{F}_p \rightarrow {}_p\widetilde H \rightarrow H \rightarrow 1\]
    and we consider the quotient $H_{i_0+1} = {}_p\widetilde H/\Phi^{i_0+1}(\widehat F_p)$. Notice that for each $i\leq i_0$, we have that $(\widehat{F}_p)_{i}/\Phi^{i_0+1}(\widehat{F}_p) \unlhd H_{i_0+1}$, and $(\widehat F_p)_{i_0}/(\widehat{F}_p)_{i_0+1} \cong F_{i_0}/F_{i_0+1}$, so the existence of a regular $\mathbb F_p \bar{H}$-module implies that $S$, being irreducible, is $H$-isomorphic to a chief factor of $H_{i_0+1}$.
\end{proof}

Notice that the chief factors of any finite quotient of ${}_p\widetilde H$ are also $H$-modules. Indeed, if $P = \ker ({}_p\widetilde H \to H)$, for any $i\geq 1$, any chief factor that appears `below' $P$ is going to be centralised by $P$, so it attains a natural ${}_p\widetilde H/P \simeq H$-module structure.

We end with a lemma. Denote by $\pi_n$ the map $H_n \to H$ coming from the inverse limit.

\begin{lemma}\label{l:o_p'}
    Suppose $H$ is a finite group whose order is divisible by $p$. Then $\pi_n(O_{p'}(H_n)) = O_{p'}(H)$ for each $n\geq 1$.
\end{lemma}
\begin{proof}
    Consider the short exact sequence
    \[1 \rightarrow M_0 \rightarrow H_1 \rightarrow H \rightarrow 1\]
    and let $N/M_0 \unlhd H_1/M_0$ be such that $N/M_0 \cong O_{p'}(H)$. By Schur-Zassenhaus, it follows that $N=M_0 \rtimes L$, with $L \cong O_{p'}(H)$, and  $O_{p'}(H)\leq C_H(M_0)$ by \cref{t:cent_frattini}, so the equality $N=M_0 \times L$ holds. The $p'$-subgroup $L$ is then normal in $H_1$, so $O_{p'}(H) = L^{\pi_1} \subseteq (O_{p'}(H_n))^{\pi_1}\subseteq O_{p'}(H)$ and the desired conclusion follows. The general case is proven by induction and the fact that $H_n$ is the universal $p$-Frattini extension of $M_{n-1}$ by $H_{n-1}$.
\end{proof}

\section{Monolithic groups}\label{sec:monolithic}

First, we deal with the easy case of monolithic groups with a non-abelian socle.

\begin{proposition}\label{p:mono_non-abelian}
	Let $G$ be a finite group with a unique minimal normal subgroup $N$. If $N$ is non-abelian, then $G$ is sectionally indecomposable.
\end{proposition}
\begin{proof}
	Let $A$ and $B$ be finite groups and let $X \leq A \times B$ be a subdirect product such that $X/K \cong G$ for some $K \unlhd X$. Denote by $\pi_A$ and by $\pi_B$ the projections of $X$ onto $A$ and $B$, respectively, and let $Y_A = \ker \pi_A$ and $Y_B = \ker \pi_B$. It is not restrictive to assume that both $Y_A$ and $Y_B$ are non-trivial. Denote by $T/K$ the minimal normal subgroup of $X/K$ and suppose that $Y_A \not\leq K$. In that case, it follows that $T/K \leq Y_AK/K$. Since $Y_A$ and $Y_B$ commute in $X$, we have $Y_BK/K \leq C_{X/K}(Y_AK/K)$. Since $Y_BK/K \leq C_{X/K}(Y_AK/K)$,  which implies that $Y_B \leq K$ and thus that $G \leq_s B$.  Thus $G$ is sectionally indecomposable by \cref{l:strong_indec_subdirect}.
\end{proof}

Henceforth, let $X$ be a monolithic primitive group, and assume that $A=\soc(X)$ is abelian. Then $X\cong A\rtimes H,$ where $H$ is an irreducible subgroup of $\aut(A).$ We want to determine under which conditions $X$ is sectionally indecomposable.

If $H=1,$ then $X$ is cyclic of prime order and $X$ is clearly sectionally indecomposable. So we may assume $H\neq 1.$ 
To simplify our arguments, the following definition will be useful.

\begin{definition}\label{defequiv} Let $H$ and $G$ be finite groups and let $A$ and $B$ be an $H$-module and a $G$-module, respectively. We say that $B$ is \emph{$H$-equivalent} to $A$ if there exist group isomorphisms $\alpha: B\to A$ and $\beta: G/C_G(B)\to H$ such that $(b^g)^\alpha=(b^\alpha)^{(gC_G(B))^\beta}$ for every $b\in B$
and $g\in G.$  In that way, $B$ can be seen as an $H$-module which is $H$-isomorphic to $A$.
\end{definition}

The motivation for \cref{defequiv} is clarified by the following lemma.

\begin{lemma}\label{eqandcf}Let $X=A \rtimes H$, with $A$ a faithful irreducible $H$-module, let $G$ be a finite group and let $B$ be an irreducible $G$-module.  Then $L_B\cong X$ if and only if $B$ is $H$-equivalent to $A$. In particular, $G$ admits an epimorphic image isomorphic to $X$ if and only if $G$ has a non-Frattini chief factor $H$-equivalent to $A$.
\end{lemma}
\begin{proof}
    First, assume that $B$ is $H$-equivalent to $A$. Given $\alpha\colon B\to A$ and
    $\beta\colon G/C_G(B)\to H$ as in \cref{defequiv}, we can define an isomorphism
    $\gamma\colon L_B=B\rtimes G/C_G(B)\to X$ by setting $(b,gC_G(B))^\gamma=b^\alpha\cdot(gC_G(B))^\beta.$

    Conversely, if $\gamma\colon L_B\to X$ is an isomorphism, then the restriction $\alpha$ of $\gamma$ to $B$ induces an isomorphism between $B=\soc(L_B)$ and $A=\soc(X)$. Moreover, for every $g\in G,$ there exists a unique $h\in H$ such that $(B,gC_G(B))^\gamma=(A,h)$. If we set $(gC_G(B))^\beta=h$, we obtain an isomorphism $\beta\colon G/C_G(B)\to H$ with the property  that $(b^g)^\alpha=(b^\alpha)^{(gC_G(B))^\beta}$ for every $b\in B$
    and $g\in G.$
\end{proof}

A \emph{proper section} of a group $G$ is either a proper quotient of $G$ or a quotient of any of its proper subgroups.

\begin{lemma}\label{preparo} Suppose we are under the hypotheses of \cref{eqandcf}. Assume that $G$ has a minimal normal subgroup $N$ that is $H$-equivalent to $A$, but no proper section of $G$ admits a chief factor $H$-equivalent to $A.$
Then either $G\cong X$ or $C_G(N)\leq \Phi(G).$
\end{lemma}
\begin{proof}
	First, we claim that $N$ is the unique minimal normal subgroup of $G$. Indeed, suppose that $M$ is another minimal normal subgroup of $G$; then $[M,N]=1$, so $M \leq C_G(N)$. Now, let $\alpha\colon N \to A$ and $\beta: G/C_G(N) \to H$ be as in \cref{defequiv}. Define $\tilde{\alpha}\colon NM/M \to A$ by $(nM)^{\tilde{\alpha}} = n^\alpha$ and $\tilde{\beta}\colon (G/M)/(C_{G/M}(NM/M)) \to H$ by $(gM\cdot C_{G/M}(NM/M))^{\tilde{\beta}} = (gC_G(N))^\beta$. Then it can be easily verified that $NM/M$ is a chief factor of $G/M$ which is $H$-equivalent to $A$.
	
If $N \not\leq \Phi(G)$, then, by \cref{eqandcf}, $X$ is an epimorphic image of $G$. Since no proper section of $G$ admits a chief factor $H$-equivalent to $A$, the kernel of the epimorphism $G \to X$ must be trivial, and therefore $G \cong X$.

It remains to prove that if $N \leq \Phi(G)$, then $C_G(N) \leq \Phi(G)$.
Assume, by contradiction, that $C_G(N) \not\leq \Phi(G)$. Then there exists a non-Frattini chief factor $U/V$ of $G$ with $N \leq V < U \leq C_G(N)$. There exists a supplement $K/V$ of $U/V$ in $G/V$. Then $N$ is a minimal normal subgroup of $K$, and we may define an isomorphism $\tilde{\beta}\colon  K/C_K(N) \to H$ by setting $(kC_K(N))^{\tilde{\beta}} = (kC_G(N))^\beta$. Then $(n^k)^\alpha = (n^\alpha)^{(kC_K(N))^{\tilde{\beta}}}$ for all $n \in N$ and $k \in K$. Thus, $N$ is a chief factor of $K$, which is $H$-equivalent to $A$, in contradiction with the assumption that no proper section of $G$ admits a chief factor $H$-equivalent to $A$.
\end{proof}

\begin{definition}\label{frath}We say that an $H$-module $A$ is \emph{$H$-Frattini} if there exist a finite group $G$ and a normal subgroup $N$ of $G$ with the following properties:
\begin{enumerate}
\item $N$ is $H$-equivalent to $A$;
\item $C_G(N)\leq \Phi(G).$
\end{enumerate}
\end{definition}

\begin{thm}\label{t:char-frath}
    Let $X=A\rtimes H$ be a monolithic primitive group with abelian socle $A$, and assume $H\neq 1.$ Then $X$ is sectionally indecomposable if and only if $A$ is not a Frattini $H$-module.
\end{thm}
\begin{proof}
First, we prove that if $A$ is a Frattini $H$-module, then $X$ is not sectionally indecomposable. Take $G$ as in \cref{frath} and let $C=C_G(N)$. For every $U\leq G,$
we define $${\rm{diag}}(U)=\{(u,u)\mid u\in U\}\leq G^2.$$
Consider $$Y=\frac{N^2 {\rm{diag}}(G)}{{\rm{diag}}(C)}.$$
Let $$M=\frac{{\rm{diag}}(C)N^2}{{\rm{diag}}(C)},\quad K=\frac{{\rm{diag}}(G)}{{\rm{diag}}(C)}.$$ Clearly $K\cong G/C \cong H$, $M\cong N$, $MK=Y$
and $M\cap K=1.$ Moreover, if $m=(n,1){\rm{diag}}(C)$ and $k=(g,g){\rm{diag}}(C),$ then $m^k=(n^g,1){\rm{diag}}(C),$ so $M$ is $H$-equivalent to $A$
and therefore $Y=M\rtimes K \cong X.$ 

We want to prove that no section of $G$ can be isomorphic to $X$. Assume by contradiction that $R/T$ is a section of $G$ with $R/T\cong X$. In particular  $U/T=\soc(R/T)$ is a chief factor of $R$ and  $R/C_R(U/T)\cong H.$ Since $C\leq \Phi(G)$ is nilpotent,
$C\cap R\leq {\rm{Fit}}(R)$ and therefore, by \cite[5.2.9]{rob},
$C\cap R\leq C_R(U/T).$
This implies
$$|H|=\left|\frac{R}{C_R(U/T)}\right|\leq \left|\frac{R}{R\cap C}\right|=\left|\frac{RC}{C}\right|\leq \left|\frac{G}{C}\right|=|H|.$$
In particular $G=RC$ and $C_R(U/T)=C\cap R$. However, since $C\leq \Phi(G),$ from $G=RC$ we deduce that $G=R$ and $C_G(U/T)=C.$ Since $U\leq C$ and $|G/U|=|G/C|=|H|,$ we deduce $U=C\leq \Phi(G)$ in contradiction with the fact that $U/T$ is complemented in $G/T.$ We have thus proved that $X$ is not sectionally indecomposable.

Conversely, assume that $X$ is not sectionally indecomposable. 
So there exist $Y_1, Y_2$ such that $X$ is neither a section 
of $Y_1$ nor of $Y_2,$ yet there exists a section $U/V$ of 
$Y_1\times Y_2$ with $U/V\cong X.$ For $i\in\{1,2\},$ 
consider the projection $\pi_i\colon Y_1\times Y_2\to Y_i.$ 
We choose $U\leq Y_1\times Y_2$ minimal with respect to 
the property that $U/V\cong X$ for some normal subgroup 
$V$ of $U.$ Without loss of generality we may assume 
$U^{\pi_1}=Y_1$ and $U^{\pi_2}=Y_2.$ Let $K_1=\ker \pi_1\cap U$ 
and $K_2=\ker \pi_2\cap U.$ It must be that $K_i\neq 1$, for $i\in\{1,2\},$
otherwise $U\cong U^{\pi_i}$ would be a subgroup of $Y_i$ with a section isomorphic to $X$.
We may identify $A$ with a chief 
factor of $U$. It follows that $V$ contains $R=R_U(A).$ 
In fact, the minimality of $U$ ensures that $V=R,$ and consequently $A\cong C/R,$ with $C=C_U(A).$
Since $(K_1\cap R)\times (K_2\cap R)$ is a normal subgroup 
of $Y_1\times Y_2$ contained in $V,$ without loss of 
generality we may assume $K_1\cap R=K_2\cap R=1.$ It follows that, for $i\in\{1,2\},$ $K_iR/R$ is a nontrivial normal subgroup of $U/R$, and therefore $C\leq K_iR.$ Since $[K_1,K_2]=1,$
$K_2\leq C_U(RK_1/R)\leq C_U(C/R)=C.$
Similarly $K_1\leq C.$
In other words, $C=R\times K_1=R\times K_2,$ and $K_1\cong_U K_2\cong_U A$
(indeed $A\cong_U C/R \cong_U RK_i/R\cong_U K_i,$ for $1\leq i\leq 2)$.
Recall that $U/R=C/R \rtimes L/R$ for a suitable subgroup $L$ of $U,$ and
$L/R\cong H.$ Now consider $J:=(K_1L)^{\pi_2}.$ First, notice that $\tilde K:=K_1^{\pi_2}\cong K_1\cong_U A$ (indeed $K_1^{\pi_2}\cong_U K_1K_2/K_2\cong_U K_2$). Moreover $C_J(\tilde K)=C^{\pi_2}$ and $J/C^{\pi_2}\cong K_1K_2L/K_2C=U/C\cong H.$
Therefore we have shown that $U$ has an epimorphic image containing a minimal normal subgroup which is $H$-equivalent to $A$. We can consider an epimorphic image of $U$ that is as small as possible with respect to this property. This epimorphic image satisfies the assumptions of \cref{preparo}, so it is either a split extension of $A$ by $H$ or $A$ is $H$-Frattini.
\end{proof}

\section{The main theorem}\label{sec:main}

Let $H$ be a finite group. While the definition of $H$-Frattini modules is needed to prove the main technical theorem of the last section, it is unclear how to check whether, for a given $H$-module $A$, there exists a finite group $G$ which satisfies the conditions of \cref{frath}. The aim of this section is to address this difficulty, and we will do so by means of the $p$-Frattini modules. The notation for this section follows the one defined in \cref{subsec:fratti_mod}.

\begin{proposition}\label{p:o_p'_trivial}
    Suppose that $A$ is a faithful irreducible non-trivial $H$-Frattini module. If $A$ is a $p$-group, then $O_{p^\prime}(H)=1.$ In particular, $H$ is not solvable.
\end{proposition}
\begin{proof}
Let $G$ be a finite group satisfying the conditions of \cref{frath}. While the extension of $C_G(N)$ by $H$ is not a $p$-Frattini extension, the fact that $N$ is a $p$-group contained in the Frattini of $G$ means that we can quotient by $O_{p'}(\Phi(G))$. Thus, it is not restrictive to assume that $C_G(N)$ is a $p$-group. The properties of the universal Frattini $p$-cover of $H$ imply the existence of an integer $n$ such that the diagram
\[\begin{tikzcd}
	1 & {O_p(H_n)} & {H_n} & {H} & 1 \\
	1 & C_G(N) & G & {H} & 1
	\arrow[from=1-1, to=1-2]
	\arrow[from=1-2, to=1-3]
	\arrow[two heads, from=1-2, to=2-2]
	\arrow["{\pi_n}",from=1-3, to=1-4]
	\arrow["{\delta}", two heads, from=1-3, to=2-3]
	\arrow[from=1-4, to=1-5]
	\arrow[r,-,double equal sign distance,double, from=1-4, to=2-4]
	\arrow[from=2-1, to=2-2]
	\arrow[from=2-2, to=2-3]
	\arrow[from=2-3, to=2-4]
	\arrow[from=2-4, to=2-5]
\end{tikzcd}\]
commutes; here, $\ker(H_n \to H) = O_p(H_n)$
 because $O_p(H)=1$. This implies $[(O_{p'}(H_n))^\delta,C_G(N)]=1$ and thus $(O_{p'}(H_n))^\delta \subseteq C_G(N)$. The commutativity and exactness of the diagram and the fact that $(O_{p'}(H_n))^{\pi_n}=O_{p'}(H)$ by \cref{l:o_p'} imply that $O_{p'}(H)=1$.
\end{proof}

\begin{proposition}\label{p:frat-mod_o-p'}
Let $A$ be a faithful irreducible $H$-module. Assume that $A$ is a $p$-group. Then $A$ is $H$-Frattini if and only if $O_{p^\prime}(H)=1.$
\end{proposition}
\begin{proof}By \cref{p:o_p'_trivial}, if $A$ is $H$-Frattini, then $O_{p^\prime}(H)=1.$

Conversely, assume that $O_{p^\prime}(H)=1.$ Recall that a group $G$ is $p$-supersolvable if and only if $G/O_{p^\prime}(G)$ has a normal $p$-Sylow subgroup such that the quotient is abelian of exponent dividing $p-1$. Since $A$ is a faithful $H$-module, $O_p(H)=1.$ Combined with $O_{p^\prime}(H)=1,$ this implies that $H$ is not  $p$-supersolvable. It follows from \cref{p:comp_factor} that $A$ is $H$-isomorphic to a chief factor in $M_n$ for some $n$; the fact that $A$ is a faithful $H$-module in turn implies that $A$ is $H$-Frattini.
\end{proof}

This yields a concrete criterion to establish whether a monolithic primitive group with abelian socle is sectionally indecomposable.

\begin{corollary}\label{c:strong_indec_o-p'}
Let $X=A\rtimes H$ be a monolithic primitive group with abelian socle $A$, and assume $H\neq 1.$ Then $X$ is sectionally indecomposable if and only if $A$ is a $p$-group and $O_{p^\prime}(H)\neq 1.$ In particular, if $H$ is solvable, then $X$ is sectionally indecomposable.
\end{corollary}

\section{Concluding remarks and open questions}\label{sec:final}

\subsection{Cohomology and irreducible modules}
Let $A$ be a faithful irreducible $H$-module. Notice that, if $O_{p'}(H) \neq 1$, it can be seen from the five-term exact sequence that $H^2(H,A)=0$, while \cref{l:H2m} indicates that if the second cohomology group is non-trivial, then the group $A \rtimes H$ is not sectionally indecomposable as long as there exists a non-split extension of $A$ by $H$ in the cohomology group which is not isomorphic to $A \rtimes H$. The following example shows that the second cohomology group being trivial is not a sufficient condition for sectional indecomposability.

We use the notation in \cite[Example 1]{GS78}. Let $H=A_5$ be the alternating group of degree 5 and let $M_0=A_2(H)$. There exist 3 irreducible $\mathbb F_2 H$-modules up to isomorphism: $\mathbb F_2$, $X$ and $P$ of dimensions 1, 4 and 4, respectively. The group $G=X \rtimes H$ is a monolithic primitive group with abelian socle and $O_{p'}(H)=1$, so $G$ is not sectionally indecomposable by \cref{c:strong_indec_o-p'}. However, from the discussion in \cite[Example 1]{GS78}, the head of $A_2(H)$ (i.e., its quotient by the radical) is isomorphic to $\mathbb F_2$, whereas the socle is isomorphic to $X$. Then \cite[Lemma 3]{GS78} on one hand implies $H^1(H,X)$ and $H^2(H,\mathbb F_2)$ are non-trivial, and on the other that $H^1(H,\mathbb F_2)$ and $H^2(H,X)$ are trivial. Furthermore, if we had considered instead $L=P \rtimes H$, then $L$ is not sectionally indecomposable and $H^1(H,P) = H^2(H,P)=0$; this illustrates the need for \cref{p:comp_factor}, since $P$ is not isomorphic to any composition factors of $M_0$, but it will be so for some $M_n$ with $n> 0$.

\subsection{On non-split extensions and $p$-groups}\label{pgruppi}

The methods presented here fall short in the attempt of proving any general statement in the context of non-primitive monolithic groups. In general, if $G$ is a monolithic group with abelian socle $A$ which does not split over $A$, one could still hope that it splits over some other abelian normal subgroup, say, $G=M \rtimes H$, in the hopes of applying \cref{l:H2}. However, even if $H^2(H,M) \neq 0$, it is not clear whether there is at least one 2-cocycle which defines a non-split extension $E$ that is not isomorphic to $G$. This prompts the following natural question, which is open to the best of our knowledge.

\begin{question}\label{q:H2}
    Let $G$ be a finite group and suppose $G=M \rtimes H$ with $M$ abelian. If $H^2(H,M) \neq 0$, does there exist a non-split extension $E$ of $M$ by $H$ corresponding to a non-trivial 2-cocycle in $H^2(H,M)$ which is not isomorphic to $G$?
\end{question}

A particular manifestation of the above is the case of $p$-groups. A $p$-group is monolithic if, and only if, its center is cyclic, and furthermore it is clearly not split over its socle. Even though we can make no general statement about the sectional indecomposability of $p$-groups, we offer some examples of $p$-groups which do not enjoy this property.

For a positive integer $n$, denote by $D_{2^n}$ and by $Q_{2^n}$ the dihedral and the generalised quaternion groups of order $2^n$, respectively. Let $n \geq 3$. Since $D_{2^n}\cong C_{2^{n-1}}\rtimes C_2,$ and $Q_{2^n}$ is a non-split extension of $C_{2^{n-1}}$ by $C_2,$ it follows from \cref{l:H2} that $D_{2^n}$ is not sectionally indecomposable.
For $G=Q_{2^n}$, consider the following construction. If
\[H=D_{2^n}=\langle x,y \mid x^{2^{n-1}}=1=y^2,\, x^y=x^{-1}\rangle\,\text{ and }
X=C_4=\langle a \mid a^4=1\rangle\]
then $G$ is neither a section of $H$ nor a section of $X.$ Take $U \leq H \times X$ generated by $(x,1)$ and $(y,a)$ and $K=\langle (x^{2^{n-2}},a^2)\rangle \unlhd U$. One can see that $U/K$ is non-abelian of order $2^n$ and $(x^{2^{n-2}},1)K$ is the unique element of order 2 in $U/K$, hence $U/K \cong Q_{2^n}.$

\medskip

Next, consider the non-abelian groups of order $p^3$ for an odd prime $p$ given by $$\begin{aligned}G_1& = (C_p\times C_p)\rtimes C_p=\langle x,y,z\mid x^p=y^p=z^p=1, [x,y]=1, x^z=xy, y^z=y\rangle,\\
G_2 &= C_{p^2} \rtimes C_p=\langle u,v\mid u^{p^2}=1, v^p=1, u^v=u^{p+1}\rangle.\end{aligned}$$
Note that $G_2$ contains a normal subgroup $N=\langle u^p, v\rangle \cong C_p\times C_p$, which is neither complemented nor central,  so $G_2$ is a non-split extension of $N$ by $C_p$ which gives a non-trivial class in $H^2(C_p,C_p \times C_p)$, so \cref{l:H2} applies and $G_1$ is not sectionally indecomposable. For $G_2$, we resort to a specific construction. Let $A=C_{p^2} = \langle a \rangle$, and  consider the subgroup $U$ of $A\times G_1$ generated by the 3 elements $\alpha = (a^p,y)$, $\beta = (a,x)$, $\gamma =(1,z)$. One can see that $\langle \alpha \rangle \leq Z(U)$ and $U/\langle \alpha \rangle$ is a non-abelian $p$-group of order $p^3$ containing an element of order $p^2$; it follows thus that $G_2 \cong U/\langle \alpha \rangle$, so $G_2$ is not sectionally indecomposable.

\medskip

Let $p$ be a prime number and let $k$ be a positive integer. We will denote the group $C_p \wr C_{p^k}$ by $W_{p,k}$; it can be equivalently rewritten as $\mathbb F_pC_{p^k} \rtimes C_{p^k}$. By \cite[Theorem 8.5]{B}, it holds that $H^2(C_{p^k},\mathbb F_pC_{p^k}) = 0$, and one can prove that the only splitting of $W_{p,k}$ as a semidirect product of a normal abelian subgroup and a complement is the one given above, so there is no hope of applying \cref{l:H2}.

\begin{thm}\label{t:cpwrcp}
For each be prime number $p$ and each positive integer $k$, the group $W_{p,k}$ is not sectionally indecomposable.
\end{thm}

\begin{proof}
    Let $p$ be a prime number.  Define the following groups:
    \begin{align*}
    G_1 = \langle f_1, \ldots, f_{p^k+1} \mid\; & f_1^{p^k}=1,\ f_i^p = 1, \ \forall\, i>1 \\
    & [f_2,f_1]=f_3,\ [f_3,f_1]=f_4,\ \ldots,\ [f_{p^k},f_1]=f_{p^k+1}, \\
    & [f_{p^k+1}, f_1]=1, \\
    & [f_i,f_j]=1, \ \forall\, i>j>1 \rangle
    \end{align*}
    \begin{align*}
    G_2 = \langle f_1, \ldots, f_{p^k+1} \mid\; &  f_2^p = f_{p^k+1},\ f_1^{p^k}=1,\ f_i^p=1 \ \text{ for }i \notin \{1,2\}, \\
    & [f_2,f_1]=f_3,\ [f_3,f_1]=f_4,\ \ldots,\ [f_{p^k},f_1]=f_{p^k+1}, \\
    & [f_{p^k+1},f_1]=1, \\
    & [f_i,f_j]=1, \ \forall\, i>j>1 \rangle
    \end{align*}
    and let $X = \langle a \mid a^{p^2} = 1 \rangle \cong C_{p^2}$.
    
    In the direct product $G_2 \times X$, let $H$ be the subgroup generated by:
    \[
   h_1 = (f_1, 1), \quad h_2=(f_2, a), \quad h_3 =(f_3, 1), \quad \ldots, \quad h_{p^k+1}=(f_{p^k+1}, 1)
    \]
    and let $K = \langle (f_2^p,\, a^p) \rangle = \langle (f_{p^k+1},\, a^p) \rangle$.
    
    Our aim is to show the following claims:
    \begin{enumerate}
    \item $K \cong C_p$,
    \item $K \leq Z(H)$,
    \item $H/K \cong G_1$.
    \end{enumerate}

    Observe first that $G_1 \simeq C_p \wr C_{p^k} = C_p^{p^k} \rtimes C_{p^k}$ and $G_2 \simeq (C_{p^2} \oplus C_p^{p^k-2}) \rtimes C_{p^k}$, so $|G_1|=|G_2|=p^{p^k+k}$. Then one observes that $G_2$ will have no normal abelian subgroups of index $p^k$ of exponent $p$, so $G_1$ and $G_2$ will not be isomorphic.  Claim (1) and Claim (2) follow directly upon inspection of the presentation of $G_2$.

    To prove Claim (3), we proceed as follows. First, the order of $H$ is $p^{p^k+k+1}$; indeed, on one hand, $H$ is a proper subgroup of $G_2 \times X$, which is of order $p^{p^k+k+2}$, and on the other, we have $|H| > |G_2|=p^{p^k+k}$, so the assertion holds. Second, let $\bar{h}_i$ be the images of the $h_i$ in $H/K$. We want to check that the relations in the presentation of $G_1$ hold for the $\bar{h}_i$. All but one of the relations are straightforward consequences of the relations of $G_2$, which we highlight below. One has
    \begin{align*}
        \bar{h}_2^p=\overline{(f_2,a)^p}=\overline{(f_2^p,a^p)}
    \end{align*}
  which is the identity in $H/K$. Thus, we get a surjective homomorphism $G_1 \to H/K$ given by $f_i \mapsto \bar{h}_i$, which is an isomorphism in light of the fact that $|H|/|K|= p^{p^k+k+1}/p = |G_1|$.
\end{proof}

All the examples of non-sectionally indecomposable $p$-groups given above follow a similar blueprint: we find another group with equal order and a ``similar'' presentation, we take the direct product with a convenient cyclic group and then we quotient by a cyclic central subgroup. We suspect that there does \emph{not} exist a non-cyclic monolithic sectionally indecomposable $p$-group, but we do not know how to generalise the above mentioned strategy to any $p$-group.

\begin{question}\label{q:p-groups}
    Do non-cyclic monolithic sectionally indecomposable $p$-groups exist?
\end{question}

\subsection{Infinite groups}

We briefly discuss the notion of sectional indecomposability for infinite groups. The following gives a plethora of examples of sectional indecomposability.

\begin{proposition}
    Let $G$ be a infinite group.
    \begin{enumerate}[(i)]
        \item If $G$ is  has a non-abelian minimal normal subgroup and its centraliser in $G$ is trivial, then $G$ is sectionally indecomposable.
        \item If $G$ is countable and contains a non-abelian free subgroup, then $G$ is sectionally indecomposable.
    \end{enumerate}
\end{proposition}
\begin{proof}
    For (i), if $G$ is contains a unique non-abelian minimal normal subgroup $N$ with $C_G(N)=1$, then we can simply mimic the proof of \cref{p:mono_non-abelian}.

    For (ii), first let $F$ be a non-abelian free group and suppose $F \leq_s A \times B$. Since $F$ is projective, it follows that $F$ is a subgroup of $A \times B$, thus we may assume that $A$ and $B$ are quotients of $F$. Consider the projections $\pi_1$ and $\pi_2$ of $F$ onto $A$ and $B$, respectively. If $\ker \pi_1$ and $\ker \pi_2$ were both non-trivial, this would mean one can find two commuting elements in $F$ that generate a free abelian group of rank 2, a contradiction. So either $A$ or $B$ is isomorphic to $F$.

    Second, suppose a countable group $G$ contains a non-abelian free subgroup $F$ and suppose $G \leq_s L \times K$. This implies that $F \leq_s L \times K$, so by the previous paragraph we may assume without loss of generality that $F \leq L$. The group $[F,F]$ is free on countably many generators, so it admits $G$ as an epimorphic image. It then follows that $G \leq_s L$.
\end{proof}

For instance, it follows that Thompson's groups $F$, $T$ and $V$ are sectionally indecomposable. Indeed, $T$ and $V$ are simple, and $[F,F]$ is a simple, normal subgroup of $F$ which is contained in every other non-trivial normal subgroup of $F$ and such that $C_F([F,F])=1$.

In contrast with the finite case, where the only sectionally indecomposable abelian groups are the cyclic ones, there do exist infinite abelian groups which are sectionally indecomposable. Let $\mathbb Z_{p^\infty}$ be the $p$-Pr{\"u}fer group for a prime $p$.

\begin{proposition}\label{p:prufer}
    $\mathbb Z_{p^\infty}$ is sectionally indecomposable.
\end{proposition}
\begin{proof}
For (i), assume $\mathbb Z_{p^\infty}\leq_s Y_1\times Y_2$.  So there exist $K\unlhd H\leq Y_1\times Y_2$ with $H/K\cong \mathbb Z_{p^\infty}$. Let $\rho$
be the restriction to $H$ of the projection $\pi_1: Y_1\times Y_2\to Y_1.$ Consider the normal subgroup $N=K\ker \rho$ of $H$. The factor group $H/N$ is an epimorphic image of $H/K\cong \mathbb Z_{p^\infty}$ so we have only two possibilities:
\begin{enumerate} 
\item $H/N\cong \mathbb Z_{p^\infty}.$ In this case $\mathbb Z_{p^\infty}$ is an epimorphic images of $H/\ker\rho \leq Y_1,$ so $\mathbb Z_{p^\infty}\leq_s Y_1.$
\item $H=N.$ In this case $H/K=K\ker\rho/K \cong \ker\rho/(K\cap\ker\rho).$ Since $\ker\rho\leq Y_2,$ we conclude $\mathbb Z_{p^\infty}\leq_s Y_2.$\qedhere
\end{enumerate}
\end{proof}

As far as divisible groups go, the Pr{\"u}fer groups are the only sectionally indecomposable ones. For a directly indecomposable divisible group $A$, if it is not isomorphic to the a Pr{\"u}fer group, it is isomorphic to $\mathbb Q$, which is an epimorphic image of $\mathbb Z_{(2)} \times \mathbb Z [1/2]$ via the map $(x,y) \mapsto x+y$, where the first factor is the additive group of the localisation of $\mathbb Z$ on the prime ideal $(2)$ and the second factor is the additive group of the subring of $\mathbb Q$ generated by $1/2$. More generally, if $\pi$ is any set of primes with $|\pi| \geq 2$, a similar argument shows that the additive subgroup of the subring of $\mathbb Q$ generated by $1/p$ for $p \in \pi$ is not sectionally indecomposable. However, when $|\pi|=1$, we do not know if sectional indecomposability holds or not.

We also provide virtually abelian examples. Let $D_{p^\infty}$ denote the semidirect product $\mathbb Z_{p^\infty} \rtimes C_2$, where $C_2$ acts by inversion.

\begin{proposition}
    $D_{p^\infty}$ is sectionally indecomposable if and only if $p$ is odd.
\end{proposition}
\begin{proof}
First, assume that $p$ is odd and $D_{p^\infty}\leq_s Y_1\times Y_2$. As such, there exist $K\unlhd H\leq Y_1\times Y_2$ with $H/K\cong D_{p^\infty}$. We use the same notation as in the proof of \cref{p:prufer} and we recall that $\ker \rho = H \cap Y_2$. We have three possibilities for the factor group $H/N$: $H=N,$ $H/N \cong D_{p^\infty},$ $H/N \cong C_2.$ In the first two cases we can conclude, as in \cref{p:prufer}, that either $D_{p^\infty}\leq_s Y_1$ or  $D_{p^\infty}\leq_s Y_2.$ Assume $H/N\cong C_2.$
In this case $N/K \cong \mathbb Z_{p^\infty}$ and
there exists $h=(y_1,y_2)\in H$ such that $n^h\equiv n^{-1}\pmod K$ for every $n\in N.$ Since $N/K = K\ker\rho/K\cong_H \ker\rho/(K\cap \ker\rho)$ the element $y_2$ acts on $\ker \rho$ as the inversion modulo $K\cap \ker \rho.$ Consider $X=\langle y_2,\ker\rho\rangle\leq Y_2$, let $\overline X=X/(K\cap \ker \rho)$ 
and $x=y_2(K\cap \ker \rho)\in \overline X.$ Notice that $\overline X=A\langle x\rangle,$ where $A=\ker\rho/(K\cap \ker \rho) \cong \mathbb Z_{p^\infty}$
and $a^x=-a$ for every $a\in A.$ Let $Z=\langle x^2\rangle.$ If $Z$ is infinite, then $Z \cap A=1,$
and $X/Z\cong D_{p^\infty}.$ Otherwise $|Z\cap A|=p^t$ for some $t\in \mathbb N$ and $|x|=2p^t$, so $X=A\langle x^{p^t}\rangle\cong D_{p^\infty}.$

If $p=2,$ we may consider the direct limit $Q_{2^\infty}$ of the direct system of the quaternion groups $Q_{2^n}$. Indeed $Q_{2^\infty}=A\langle x\rangle,$ where $A\cong \mathbb Z_{2^\infty},$
$a^x=-a$ for every $a\in A$ and $x^2$ coincides with the unique element of order 2 in $A.$ By the same argument used to see that $D_{2^n}\leq_s Q_{2^n}\times Q_{2^n},$  we can deduce that $D_{2^\infty}\leq_s Q_{2^\infty}\times Q_{2^\infty}.$
\end{proof}

A last remark on the non-abelian case is that there are many interesting test cases which are not covered by any of the above results. For instance, periodic branch groups, which are far from being simple since they are residually finite and do not contain non-abelian free subgroups. However, it is not clear if there should be any sort of general criterion to prove or disprove sectional indecomposability for them, so we do not explore this direction in this work.

\bibliographystyle{abbrv}
\bibliography{References}

\end{document}